\documentclass[12pt]{amsart}

\usepackage[a4paper,margin=1in]{geometry}
\usepackage{amsmath,amssymb,amsthm,mathtools}
\usepackage{enumitem,orcidlink,url}

\newtheorem{theorem}{Theorem}[section]
\newtheorem{proposition}[theorem]{Proposition}
\newtheorem{lemma}[theorem]{Lemma}

\newcommand{\R}{\mathbb{R}}

\newcommand{\Q}{\mathbb{Q}}
\newcommand{\Z}{\mathbb{Z}}
\newcommand{\Res}{\operatorname{Res}}

\title[Algebraic and analytic structure of Morikawa's problem]{Algebraic and analytic structure of Morikawa's sangaku problem}
\author[D. Krumm]{David Krumm\,\orcidlink{0000-0001-9999-2166}}
\address{Texas Tech University--Costa Rica, San José, Costa Rica}
\email{dkrumm@ttu.edu}
\urladdr{http://maths.dk}


\begin{document}

\begin{abstract}
Let $\mu(r)$ denote the minimal side length of a square inscribed in the curvilinear triangular region formed by two tangent circles of radii $1$ and $r \ge 1$ together with their common tangent line. The problem of finding a closed-form expression for $\mu(r)$ was posed in early nineteenth-century Japan by Morikawa. It was proved by Holly and Krumm~\cite{HollyKrumm2021} that no expression in radicals exists for $\mu(r)$. In this article we show that $\mu$ is an algebraic function, and consequently real-analytic on $[1,\infty)$ outside a finite explicitly computable set. In particular, although no expression in radicals exists, the function admits convergent Taylor expansions at all non-exceptional values of $r$, whose coefficients may be computed by Newton iteration from the defining algebraic equation. We illustrate the method by explicitly computing the Taylor expansion of $\mu(r)$ at $r=1$.
\end{abstract}

\maketitle

\section{Introduction}\label{sec:intro}

In Edo-period Japan (1603-1868), \emph{sangaku} were wooden tablets on which geometric problems were inscribed, and which were commonly dedicated at temples or shrines. These problems, often highly sophisticated, reflect a rich tradition of mathematics that developed during a period of limited international contact. A primary source of information on sangaku of the Edo period is a diary written in the early nineteenth century by the mathematician Yamaguchi Kanzan, who traveled extensively across Japan recording the contents of such tablets. Yamaguchi's diary is preserved in the city of Agano, Niigata Prefecture, as a designated cultural asset. Excerpts from the diary appear in translation in the book of Fukagawa and Rothman \cite{FukagawaRothman2016}, whose authors remark that: ``most [of the problems], we concede, are extremely difficult and one or two remain unsolved to the present day.'' \cite[p.~245]{FukagawaRothman2016}

Of the problems said to remain unsolved, one was posed by Sawa Masayoshi in 1821 and solved by Kinoshita \cite{Kinoshita2018} in 2018. The other, posed by Morikawa Jihei during the same time period, was shown by Holly and Krumm \cite{HollyKrumm2021} not to admit the kind of explicit solution traditionally sought. Morikawa's problem is easily stated by reference to Figure~\ref{fig:lens}. Two circles, $C_1$ and $C_r$, of radii $1$ and $r\ge1$, respectively, are tangent to each other and to a common horizontal line $L$. Among all squares contained in the resulting curvilinear triangular region and touching each of $C_1$, $C_r$, and $L$, let $\mu(r)$ denote the minimal possible side length. The problem asks for an explicit determination of $\mu(r)$ as a function of $r$.

\begin{figure}[ht]
  \centering
  \includegraphics[width=0.6\textwidth]{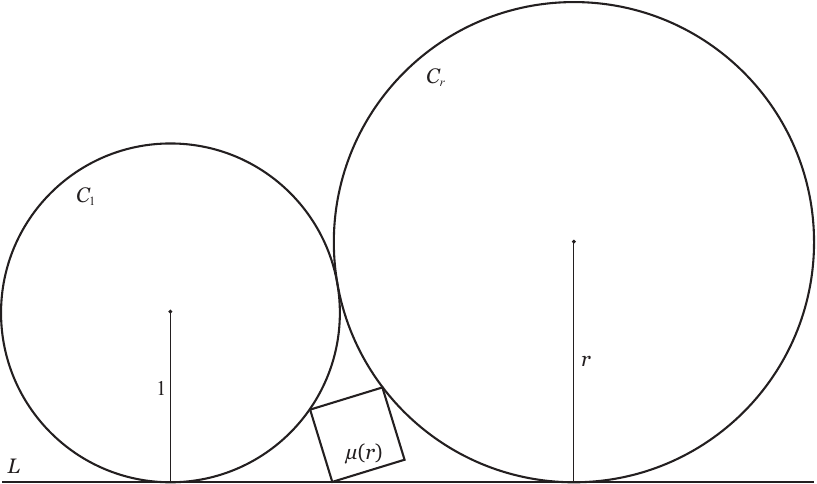}
  \caption{The region determined by two tangent circles of radii $1$ and $r\ge1$
  together with their common tangent line, and an inscribed square of minimal side
  length $\mu(r)$.}
  \label{fig:lens}
\end{figure}

Morikawa did not specify the precise form of expression sought for $\mu(r)$. However, the mathematical culture in which the problem arose strongly emphasized explicit algebraic expressions; it is therefore natural to interpret the question as asking for a formula obtainable by radicals. Under this interpretation, Morikawa’s problem was answered in \cite{HollyKrumm2021}, where it was proved that $\mu(r)$ is not expressible in radicals.

The absence of a solution by radicals does not, however, preclude other forms of explicit description. By the early modern period in Europe, techniques based on implicit functions, power series, and algebraic curves had been developed to describe solutions of equations that resist solution by radicals. In light of these developments, Morikawa’s function $\mu(r)$ may still admit an explicit analytic description even though it does not admit a closed-form algebraic representation.

The purpose of the present article is to make explicit the algebraic structure of the function $\mu(r)$, and to analyze the analytic consequences of that structure. In particular, we prove that $\mu$ is real-analytic outside a finite exceptional set.

\begin{theorem}\label{thm:analyticity}
There exists a finite set $E\subset[1,\infty)$ of real algebraic numbers such that the function
$\mu\colon [1,\infty)\setminus E \longrightarrow\R$ is real-analytic.
\end{theorem}

Thus, $\mu(r)$ can be locally described by convergent power series at all but finitely many values of $r$. Our approach is based on identifying an explicit algebraic equation satisfied by $\mu$, and then analyzing its discriminant in order to determine where the real-analytic implicit function theorem may be applied. The obstruction to radicals discovered in \cite{HollyKrumm2021} is therefore only one manifestation of a richer structure: the function $\mu$ is globally constrained by a polynomial equation in the variables $(r,\mu)$, and locally governed by analytic expansions wherever the discriminant does not vanish. In this sense, Morikawa’s problem admits a modern analytic resolution: although no radical formula exists, the function enjoys the full analytic regularity permitted by its algebraic nature.

The power series expansions of $\mu$ are not merely guaranteed to exist abstractly: once the defining algebraic equation is known and a non-exceptional radius is chosen, the coefficients of the power series can be computed by applying Newton iteration in a power series ring. As an illustration of the method, expanding about $r=1$ yields
\[
\mu(1+t)\approx 
\frac{227}{589}
+ \frac{157}{883}t
- \frac{55}{783}t^{2}
+ \frac{33}{898}t^{3}
- \frac{10}{449}t^{4}
+ \frac{7}{473}t^{5}
+ O(t^{6})
\]
for sufficiently small $t>0$. This expansion provides a concrete and computable analogue of the kind of explicit formula that might naturally have been inscribed in a sangaku.

The paper is organized as follows. In Section~\ref{sec:algebraic}, we use polynomial resultants to construct an explicit polynomial relation satisfied by the auxiliary function $\lambda(k)=(\mu(k^2))^2$, thus proving that $\mu$ is algebraic. In Section~\ref{sec:analyticity}, we analyze the discriminant of this polynomial to identify a finite exceptional set and apply the real-analytic implicit function theorem to establish Theorem~\ref{thm:analyticity}. Finally, in Section~\ref{sec:taylor}, we make the analytic structure concrete by computing a Taylor expansion of $\mu$ at $r=1$ using Newton iteration in the power series ring $\R[[t]]$, and we compare the resulting approximation with independent numerical computations using the methods of \cite{HollyKrumm2021}.

All computations described in the paper were carried out using Sage~\cite{SageMath}. Complete source code is available in~\cite{KrummMorikawaCode}.

\section{An algebraic equation for the Morikawa function}\label{sec:algebraic}

Recall that a real-valued function $f$ defined on an interval $I\subset\R$ is called \emph{algebraic} if there exists a nonzero polynomial $P(x,y)\in\R[x,y]$ such that $P(c,f(c))=0$ for all $c\in I$. We show here that the function $\mu$ is algebraic.

To simplify the analysis, we introduce the auxiliary function
\[\lambda(k) := (\mu(k^2))^2, \qquad k\ge 1.\]
The change of variables $r=k^2$ is a real-analytic bijection of $[1,\infty)$ onto itself, with real-analytic inverse $k=\sqrt{r}$. Consequently, $\mu$ is real-analytic at $r_0$ if and only if $\lambda$ is real-analytic at $k_0=\sqrt{r_0}$. Moreover, it is clear from the definition of $\lambda$ that $\mu$ is algebraic if and only if $\lambda$ is. We therefore carry out the analysis using $\lambda$ and translate our main conclusions back to $\mu$.

\begin{proposition}\label{prop:lambda_algebraic}
The function $\lambda$ (and therefore $\mu$) is algebraic on $[1,\infty)$.
\end{proposition}
\begin{proof}
We build on the algebraic relations derived by Holly and Krumm~\cite{HollyKrumm2021}. Let $p(k,x)\in\Z[k,x]$ be the polynomial defined by \cite[Eq.~(2), p.~19]{HollyKrumm2021}. In the paragraph following the proof of \cite[Lemma~5.2]{HollyKrumm2021}, the authors introduce a function $\xi(k)$ and show that
\[
p(c,\xi(c))=0
\qquad\text{for every }c\in[1,\infty);
\]
see \cite[Eq.~(4)]{HollyKrumm2021}. They then construct a polynomial $h\in\Z[k,x,y]$ such that
\[
h(c,\xi(c),\lambda(c))=0
\qquad\text{for every }c\in[1,\infty);
\]
see \cite[Eq.~(5)]{HollyKrumm2021}. It follows that for every $c\ge 1$, the number $\xi(c)$ is a common root of the polynomials $p(c,x)$ and $h(c,x,\lambda(c))$, hence their resultant vanishes. Defining
\[
R(k,y)=\Res_x\bigl(p(k,x),\,h(k,x,y)\bigr)\in\Z[k,y],
\]
we therefore have $R(c,\lambda(c))=0$ for all $c\ge 1$. Computing $R(k,y)$ shows that it is not the zero polynomial, and thus $\lambda$ is an algebraic function on $[1,\infty)$, as claimed.

For later use, we derive an equation of smaller degree satisfied by $\lambda$. Factoring $R(k,y)$, we obtain \[R(k,y)=2^{26}(k^2+1)^2P(k,y),\] where $P\in\Z[k,y]$ is a primitive polynomial of total degree 144. Since $c^2+1\neq 0$ for every real $c$, it follows that $P(c,\lambda(c))=0$ for $c\ge 1$, again showing that $\lambda$ is algebraic.
\end{proof}

\section{Analyticity outside a finite set}\label{sec:analyticity}

Having established that $\lambda$ satisfies a polynomial relation $P(c,\lambda(c))=0$, we now investigate its local analytic behavior. The first step in this analysis is to prove that $\lambda$ is continuous.

Recall from \cite[Prop.~4.2]{HollyKrumm2021} that for each $r\ge1$, $\mu(r)$ is the minimum value of the function
\[z_r(x):=\sqrt{x^{2}+\Bigl(r-x-\sqrt{\,r^{2}-\bigl(2\sqrt r-x-\sqrt{2x-x^{2}}\bigr)^{2}}\,\Bigr)^{2}}\]
on the interval $\bigl(1-1/\sqrt2,\,1\bigr)$. Moreover, $z_r$ has a unique minimizer $x_m=x_m(r)$ in this interval, and is strictly decreasing for $x<x_m$ and strictly increasing for $x>x_m$. A straightforward verification using elementary inequalities shows that every term under a square root in the expression for $z_r(x)$ remains nonnegative on $[1-1/\sqrt2,1]$, so $z_r$ extends continuously to this closed interval.

The above formulation serves both to establish continuity of $\mu$ and to facilitate numerical computation, since $\mu(r)$ can be approximated by minimizing $z_r(x)$ on $[1-1/\sqrt2,1]$ using standard one-variable optimization routines.

\begin{lemma}\label{lem:mu_continuous}
The function $\mu$ (and therefore $\lambda$) is continuous on $[1,\infty)$.
\end{lemma}
\begin{proof}
Let $I=[1-1/\sqrt2,1]$ and fix $R>1$. The map $(r,x)\mapsto z_r(x)$ is continuous on the compact set
$[1,R]\times I$, hence uniformly continuous there. Thus, given $\varepsilon>0$ there exists
$\delta>0$ such that if $|r-s|<\delta$ (with $r,s\in[1,R]$), then
\[|z_r(x)-z_s(x)|<\varepsilon \quad\text{for all } x\in I.\]
Taking minima over $x\in I$ gives
\[|\mu(r)-\mu(s)|
=\Bigl|\min_{x\in I} z_r(x)-\min_{x\in I} z_s(x)\Bigr|
\le \max_{x\in I}|z_r(x)-z_s(x)|
<\varepsilon,
\]
so $\mu$ is continuous on $[1,R]$. Since $R$ was arbitrary, $\mu$ is continuous on $[1,\infty)$.
\end{proof}

With continuity in hand, we now recall the real-analytic implicit function theorem (see, for example, \cite[Thm.~6.1.2]{KrantzParks2013}), which will be our main tool for proving analyticity outside a finite exceptional set.

\begin{theorem}[Real-analytic implicit function theorem]\label{thm:IFT}
Let $F(x,y)$ be real-analytic in a neighborhood of $(x_0,y_0)\in\R^2$.
Suppose that
\[
F(x_0,y_0)=0
\quad\text{and}\quad
\frac{\partial F}{\partial y}(x_0,y_0)\neq 0.
\]
Then there exist open intervals $U$ containing $x_0$ and $V$ containing $y_0$
and a real-analytic function $f:U\to V$ such that for all $(x,y)\in U\times V$, 
\[F(x,y)=0\quad \Longleftrightarrow \quad y=f(x).\]
\end{theorem}

We may now prove Theorem~\ref{thm:analyticity}.

\begin{proposition}\label{prop:lambda_analytic}
There exists a finite set $E_0 \subset [1,\infty)$ of real algebraic numbers such that $\lambda$ is real-analytic on $[1,\infty)\setminus E_0$. Consequently, $\mu$ is real-analytic outside a finite subset of $[1,\infty)$.
\end{proposition}
\begin{proof}
With $P(k,y)$ as in the proof of Proposition~\ref{prop:lambda_algebraic}, define
\[\Delta(k)=\Res_y\left(P,P_y\right)\in\Z[k],\]
where $P_y=\frac{\partial P}{\partial y}$. Let $\Delta_0$ be the primitive part of the squarefree part of $\Delta$. Factoring $\Delta_0$ yields
\[\Delta_0(k)=k(k-1)(2k+1)(k^2+1)J(k),\]
where $J$ is a product of 19 irreducible polynomials in $\Z[k]$ with degrees in $[2,738]$.

Let $E_0$ be the set of real roots of $\Delta_0$ in $[1,\infty)$, and note that $E_0$ coincides with the set of real roots of $\Delta$ in $[1,\infty)$. Clearly, $E_0$ is a finite set of algebraic numbers including 1; in our computation, $E_0$ has 138 additional elements, all in the interval $(1.03, 163.04)$.

We now describe the analytic behavior of $\lambda$ away from the exceptional set $E_0$. Since $E_0$ is finite and $1\in E_0$, the set $[1,\infty)\setminus E_0$ is a disjoint union of open intervals (in fact, 139 intervals). Fix $k_0$ in any of these intervals, and set $y_0=\lambda(k_0)$. Since $k_0\notin E_0$, we have $P_y(k_0,y_0)\neq0$. By Theorem~\ref{thm:IFT}, there exist  neighborhoods $U$ of $k_0$ and $V$ of $y_0$, and a real-analytic function $f:U\to V$ such that for all $(x,y)\in U\times V$,
\[P(x,y)=0\quad \Longleftrightarrow \quad y=f(x).\]
Since $\lambda$ is continuous by Lemma~\ref{lem:mu_continuous}, after possibly shrinking $U$ we may assume
that $\lambda(c)\in V$ for all $c\in U\cap[1,\infty)$. For such $c$ we have
$P(c,\lambda(c))=0$, and hence, by the equivalence displayed above, $\lambda(c)=f(c)$. In particular, $\lambda$ is real-analytic in a neighborhood of $k_0$. This proves the proposition, as $k_0$ was an arbitrary element of $[1,\infty)\setminus E_0$.
\end{proof}

\section{Computation of a Taylor expansion}\label{sec:taylor}

We conclude by illustrating how the analytic structure of $\lambda$ and $\mu$ can be made explicit in practice. In Proposition~\ref{prop:lambda_analytic}, we showed that $\lambda$ is real-analytic outside the finite set $E_0$ of roots of the discriminant $\Delta(k)=\Res_y(P,P_y)$ in $[1,\infty)$. The vanishing of $\Delta$ at a point $k_0$ indicates that the real-analytic implicit function theorem cannot be applied directly at $(k_0,\lambda(k_0))$. However, membership in $E_0$ does not by itself imply that $\lambda$ fails to be real-analytic at $k_0$. To illustrate this phenomenon, we examine the point $1\in E_0$, which corresponds to the case where the two circles in Morikawa's problem have equal radius. We show that $\lambda$ admits a convergent power series expansion at $k=1$ and we compute its first several coefficients; from this expansion we then derive the expansion of $\mu$ at $r=1$.

\subsection*{Analyticity at $k=1$}

Building on the fact that $\lambda$ is continuous at $k=1$, we now prove it is analytic at that point. Throughout the remainder of the paper, we set
\[a:=\lambda(1)\approx 0.14853189819,\] 
the latter approximation being obtained from the minimization procedure described in Section~\ref{sec:analyticity}. Recalling that $P(c,\lambda(c))=0$ for all $c\ge 1$, we have $P(1,a)=0$. 

\begin{proposition}\label{prop:lambda_at_1}
The function $\lambda$ is real-analytic at $k=1$.
\end{proposition}
\begin{proof}
A factorization of the polynomial $P(k,y)$ shows that $P$ has a unique irreducible factor $q(k,y)$ of total degree 28; moreover, $P(k,y)=q(k,y)s(k,y)$ with $s\in\Z[k,y]$ satisfying $s(1,a)\neq 0$. It follows that $q(1,a)=0$ since $P(1,a)=0$.

Specializing $q$ at $k=1$ and factoring yields
\[q(1,y)=8(y-2)(y-4)(y-8)n(y),\]
where $n(y)\in\Q[y]$ is an irreducible polynomial of degree 7. It follows that $q(1,y)$ has no repeated roots, and in particular,
\[q(1,a)=0\quad\text{and}\quad\frac{\partial q}{\partial y}(1,a)\neq0.\]

By Theorem~\ref{thm:IFT}, there exist open intervals $U$ containing $1$ and $V$ containing $a$, and a real-analytic function $f:U\to V$ such that for all $(x,y)\in U\times V$, 
\[
q(x,y)=0
\quad\Longleftrightarrow\quad
y=f(x).
\]
Since $s(1,a)\neq0$ and $s$ is continuous, after shrinking $U$ and $V$ if necessary we may assume that
$s(x,y)\neq0$ for all $(x,y)\in U\times V$. By Lemma~\ref{lem:mu_continuous}, $\lambda$ is continuous at $1$, so after possibly shrinking $U$
we have $\lambda(c)\in V$ for all $c\in U\cap[1,\infty)$. For such $c$ we obtain
\[
0=P(c,\lambda(c))=q(c,\lambda(c))\,s(c,\lambda(c)).
\]
Since $s(c,\lambda(c))\neq 0$, it follows that $q(c,\lambda(c))=0$. Because $(c,\lambda(c))\in U\times V$, this implies $\lambda(c)=f(c)$ for all $c\in U\cap[1,\infty)$. Thus, $\lambda$ agrees near $k=1$ with a real-analytic function, and is therefore real-analytic at $k=1$.
\end{proof}

\subsection*{Taylor expansion of $\mu$}

With $n(y)$ as in the proof of Proposition~\ref{prop:lambda_at_1}, we have $n(a)=0$ since $q(1,a)=0$. Hence, $a$ is a real algebraic number (indeed $a$ has degree $7$ over $\Q$).

Set $t=k-1$ and write $F(t,y)=q(1+t,y)$.  Then
\[F(0,a)=0
\quad\text{and}\quad
\frac{\partial F}{\partial y}(0,a)\neq0.\]

Note that $F(t,y)$ is a polynomial in $y$ with coefficients in the complete local ring $\R[[t]]$. Since $F_y(0,a)\neq 0$, Hensel's lemma (\cite[Thm. 7.3]{EisenbudCA}) implies that there is a unique power series $y(t)\in \R[[t]]$ such that $y(0)=a$ and $F(t,y(t))=0$ in $\R[[t]]$. We claim that $y(t)$ is the Taylor series of $\lambda(1+t)$ at $t=0$. 

By Proposition~\ref{prop:lambda_at_1}, $\lambda$ is real-analytic at $k=1$, hence the function $t\mapsto \lambda(1+t)$ is real-analytic in a neighborhood of $t=0$. Moreover, the proof of the proposition shows that there exists $\varepsilon>0$ such that $q(c,\lambda(c))=0$ for all $c\in[1,1+\varepsilon)$. It follows that
\[F\bigl(t,\lambda(1+t)\bigr)=0\quad\text{for all}\quad t\in[0,\varepsilon).\]
Let $g(t)\in\R[[t]]$ denote the Taylor series of $\lambda(1+t)$ at $t=0$. Since both $F$ and $t\mapsto \lambda(1+t)$ are real-analytic near $t=0$, the function $t\mapsto F\bigl(t,\lambda(1+t)\bigr)$ is real-analytic near 0; moreover, it vanishes on $[0,\varepsilon)$, so its Taylor series is the zero power series.

Since $F(t,y)$ is polynomial in $y$ with real-analytic coefficients in $t$, the Taylor series of $t\mapsto F(t,\lambda(1+t))$ is obtained by replacing $\lambda(1+t)$ with its Taylor series $g(t)$, and therefore coincides with $F\bigl(t,g(t)\bigr)$. It now follows from the previous paragraph that $F\bigl(t,g(t)\bigr)=0$ in $\R[[t]]$. Since $y(t)\in\R[[t]]$ is the unique power series satisfying $y(0)=a$ and $F(t,y(t))=0$, this proves that $y(t)=g(t)$, as claimed.

Newton approximation (\cite[Chap.~XII, \S7, Prop.~7.6]{LangAlgebra}) now provides a simple method for computing the series $y(t)$, i.e., the Taylor series of $\lambda(1+t)$ at $t=0$. Concretely, $y(t)$ is the limit of the sequence $(y_n(t))$ given by
\[
y_1(t)=a,\qquad
y_{n+1}(t)=y_n(t)-\frac{F(t,y_n(t))}{F_y(t,y_n(t))}.
\]

The series $y(t)$ can thus be computed to any $t$-adic precision. We may therefore construct the Taylor series for $\mu(1+t)$ at $t=0$ by using the binomial series for the square root, recalling that $\mu(r)=\sqrt{\lambda(\sqrt r)}$. Carrying out this procedure and truncating the resulting power series yields
\[\mu(1+t)=a_0+a_1t+a_2t^2+a_3t^3+a_4t^4+a_5t^5+O(t^6),\]
where the coefficients $a_i$ have the following decimal approximations:
\begin{align*}
a_0&\approx 0.3853983629832700199,\\
a_1&\approx 0.1778035533283861916,\\
a_2&\approx -0.070242718387397787,\\
a_3&\approx 0.03674792508131597225,\\
a_4&\approx -0.022272764064999471,\\
a_5&\approx 0.01479982358063203048.
\end{align*}

For readability, we may replace each coefficient above with a rational approximation to obtain the expression for $\mu(1+t)$ displayed in Section~\ref{sec:intro}.

To assess the accuracy of this expansion, we computed the series up to order $t^{40}$ and compared the resulting approximations with values obtained independently via the minimization procedure described in Section~\ref{sec:analyticity}. The results are summarized in Table~\ref{tab:comparison}. To illustrate both the accuracy near $t=0$ and the eventual loss of accuracy away from the center, we include a value of $t$ beyond the range where the truncated series provides a good approximation.

\begin{table}[ht]
\centering
\caption{Comparison of the truncated Taylor expansion of $\mu(1+t)$
with values obtained via numerical minimization.}
\label{tab:comparison}
\begin{tabular}{ccc}
\hline
$t$ & Numerical value & Taylor approximation  \\
\hline
0.1 & 0.4025109500237806 & 0.40251095002378024\\
0.2 & 0.4184118635505329 & 0.41841186355053286\\
0.5 & 0.4602836437482523 & 0.46028364374825212\\
1.0 & 0.5161758482795963 & 0.51627457893057584\\
1.2 & 0.5350575163508569 & 0.66732178296134338\\ 
\hline
\end{tabular}
\end{table}


\end{document}